%% file: cyclic_quadrilaterals.tex
\documentclass[10pt]{amsart}
\input{macros/preamble.tex}

\begin{document}
\thispagestyle{empty}
\title[Cyclic quadrilaterals and smooth Jordan curves]{Cyclic quadrilaterals and smooth Jordan curves}
\author{Joshua Evan Greene} 
\address{Department of Mathematics, Boston College, USA}
\email{joshua.greene@bc.edu}
\urladdr{https://sites.google.com/bc.edu/joshua-e-greene}
\author{Andrew Lobb} 
\address{Mathematical Sciences,
	Durham University,
	UK}
\email{andrew.lobb@durham.ac.uk}
\urladdr{http://www.maths.dur.ac.uk/users/andrew.lobb/}
\thanks{JEG was supported on NSF Award DMS-2005619.}

\begin{abstract}
For every smooth Jordan curve $\gamma$ and cyclic quadrilateral $Q$ in the Euclidean plane, we show that there exists an orientation-preserving similarity taking the vertices of $Q$ to $\gamma$.
The proof relies on the theorem of Polterovich and Viterbo that an embedded Lagrangian torus in $\C^2$ has minimum Maslov number 2.
\end{abstract}

\maketitle

\input{sections/main_result.tex}

\subsection*{Acknowledgements.}
\input{sections/acknowledgements.tex}

\bibliographystyle{amsplain}
\bibliography{references/works-cited.bib}
\end{document}

%% file: macros/preamble.tex
\usepackage{amsmath} \usepackage{amssymb, amscd,cancel, graphicx,soul,stmaryrd}
	
\usepackage{pstricks,pinlabel,mathtools}

\usepackage{enumerate}

\usepackage{tikz-cd}

\usepackage{mathdots}
\newtheorem{thm}{Theorem}

\newtheorem{lem}[thm]{Lemma}

\newtheorem*{ques*}{Question}

\newtheorem*{example*}{Example}
\newtheorem*{porism*}{Porism}
\newtheorem*{scholium*}{Scholium}

\newtheorem*{thm*}{Theorem}
\newtheorem*{defin*}{Definition}
\newtheorem*{lem*}{Lemma}
\newtheorem*{prop*}{Proposition}

\def\C{{\mathbb C}}

\def\bR{{\mathbb R}}

\def\bZ{{\mathbb Z}}

\def\del{{\partial}}



%% file: sections/main_result.tex
A quadrilateral $Q$ {\em inscribes} in a smooth Jordan curve $\gamma$ in the Euclidean plane if there exists an orientation-preserving similarity of the plane taking the vertices of $Q$ to $\gamma$; it is {\em cyclic} if it inscribes in a circle.
The result of this paper is the solution of the cyclic quadrilateral peg problem \cite[Conjecture~9]{matschke2014}:

\begin{thm*}
	\label{thm:main_theorem}
	Every cyclic quadrilateral inscribes in every smooth Jordan curve in the Euclidean plane.
\end{thm*}

\noindent
The result is best possible, by considering the case in which the smooth Jordan curve is itself a circle.
Moreover, some regularity hypothesis on the Jordan curve is necessary in order for the Theorem to hold, as the only cyclic quadrilaterals that inscribe in all triangles are the isosceles trapezoids \cite[$\S$~3.6]{pak2008}.

\begin{proof}
For a fixed cyclic quadrilateral $Q$ and smooth Jordan curve $\gamma$, we construct a pair of Lagrangian tori $T_1$ and $T_2$ in standard symplectic $\C^2$.
They intersect cleanly along $\gamma \times \{0\}$ and in a disjoint set of points $P$ which parametrize the inscriptions of $Q$ in $\gamma$.
By smoothing the intersection along $\gamma \times \{0\}$, we obtain an immersed Lagrangian torus $T$ whose set of self-intersections is $P$.
As we show, $T$ has minimum Maslov number 4.
On the other hand, a theorem independently due to Polterovich and Viterbo asserts that an embedded Lagrangian torus in $\C^2$ has minimum Maslov number 2 \cite{polto1991,viterbo1990}.
Therefore $P$ is non-empty, so $Q$ inscribes in $\gamma$.
\end{proof}

The strategy of proof of the Theorem resembles that of our earlier result, which treated the case in which $Q$ is a rectangle \cite{greenelobb1}.  In that case, we additionally arranged that $T$ is invariant under a symplectic involution $\tau$ of $\C^2$.
Passing to the quotient by $\tau$, we obtained an immersed Lagrangian Klein bottle $K = T / \tau$ in $\C^2$ whose self-intersections $P / \tau$ parametrize inscriptions of $Q$ in $\gamma$ up to rotation by $\pi$.
A theorem independently due to Shevchishin and Nemirovski asserts that there is no embedded Lagrangian Klein bottle in $\C^2$  \cite{nemirovski2,shevchishin2009}, thereby ensuring that $P$ is non-empty, so $Q$ inscribes in $\gamma$.
In the more general case of a cyclic quadrilateral, $T$ does not admit any apparent symmetry, which impedes reusing the same approach.  
Our revised approach produces a stronger result and somewhat more directly.

\noindent
{\bf Cyclic quadrilaterals.}
We begin by characterizing the set of cyclic quadrilaterals.
Let $Q$ denote a convex quadrilateral in the plane whose vertices are labeled $ABCD$ in counterclockwise order.
Its diagonals $AC$ and $BD$ intersect in a point $X$.
Euclid's chord theorem asserts that
$Q$ is cyclic if and only if $|AX| \cdot |CX| = |BX| \cdot |DX|$ \cite[Theorem III.35]{euclid}.\footnote{Euclid proves the forward direction, which can be used to prove the reverse.}

By a cyclic permutation of the vertex labels, we may assume that $|AX| \le |CX|$ and $|BX| \le |DX|$.
We thereby obtain real values $s = |AX|/|AC|$ and $t = |BX|/|BD|$ in $(0,1/2]$ and an angle $\phi = \angle AXB$ in $(0,\pi)$.
The triple of values $(s,t,\phi)$ uniquely determines the oriented similarity class of $Q$, unless one of $s$ and $t$ equals $1/2$, in which case $(s,t,\phi)$ and
$(t,s,\pi-\phi)$ determine the same oriented similarity class.

We reformulate the preceding description for our present purposes.
Identify the Euclidean plane with the complex numbers $\C$.
Define $\C$-linear automorphisms of $\C^2$ by the matrices
\[
F_r = \left( \begin{matrix} r & 1-r \\ \sqrt{r(1-r)} & -\sqrt{r(1-r)} \end{matrix} \right) \quad \textup{and} \quad R_\phi = \left( \begin{matrix} 1 & 0 \\ 0 & e^{i \phi} \end{matrix} \right)
\]
for values $r \in (0,1/2]$ and $\phi \in (0,\pi)$.

\begin{lem}
	\label{lem:parametrizing_cyclic_quads}
Points $A,B,C,D \in \C$ correspond as above to vertices of a cyclic quadrilateral with parameters $(s,t,\phi)$ if and only if
\begin{equation}
\label{eq: cyclic quad}
R_\phi \circ F_s(A,C) = F_t(B,D) \quad \textup{ and } \quad A \ne C \,\, ({\rm equivalently} \,\,\, B \ne D).
\end{equation}
\end{lem}
\begin{proof}
Equality in the first coordinate of \eqref{eq: cyclic quad} is equivalent to the assertion that segments $AC$ and $BD$ intersect at a point $X$ so that $|AX| = s \cdot |AC|$ and $|BX| = t \cdot |BD|$.
Equality in the second coordinate given the first then ensures that $\angle AXB = \phi$ and that $|AX| \cdot |CX| = s(1-s) \cdot |AC|^2 = t(1-t) \cdot |BD|^2 = |BX| \cdot |DX|$.
Insisting that $A \ne C$ or $B \ne D$ ensures that $Q$ does not degenerate to a point.
\end{proof}

\noindent
{\bf Two embedded Lagrangian tori.}
Suppose that $Q$ is a cyclic quadrilateral with parameters $(s,t,\phi)$ as above and that $\gamma$ is a smooth Jordan curve in $\C$.
Note that $\gamma \times \gamma$ is a smoothly embedded torus in $\C^2$.  Define tori 
\[
T_1 = R_\phi \circ F_s ( \gamma \times \gamma) \quad \textit{and} \quad T_2 = F_t(\gamma \times \gamma).
\]
Note that both $R_\phi \circ F_s$ and $F_t$ map the point $(z,z)$ to $(z,0)$ for all $z \in \C$.
From Lemma \ref{lem:parametrizing_cyclic_quads} we see that the set of inscriptions of $Q$ in $\gamma$ is parametrized by the set of points
\[
P = T_1 \cap T_2 - \gamma \times \{0\}.
\]
Let $\omega = dz \wedge d\overline{z} + dw \wedge d\overline{w}$ denote the standard symplectic form on $\C^2$, up to scale.

\begin{lem}
	\label{lem:lagrangianandclean}
	The tori $T_1$ and $T_2$ are Lagrangian with respect to $\omega$ and intersect cleanly along $\gamma \times \{0\}$:
	\[
	T_{(p,0)}T_1 \cap T_{(p,0)}T_2 = T_{(p,0)} (\gamma \times \{0\}), \, \, \, \textit{for all } p \in \gamma.
	\]
\end{lem}

\begin{proof}
A direct calculation shows that
\[
\omega_r : = F_r^* \, \omega = r \cdot dz \wedge d\overline{z} + (1-r) \cdot dw \wedge d\overline{w}
\]
for $r \in (0,1/2]$.
Note that $\gamma \times \gamma$ is Lagrangian with respect to $\omega_r$ and $R_\phi^* \, \omega = \omega$.
It follows that $T_1$ and $T_2$ are Lagrangian with respect to $\omega$.

If $p \in \gamma$ is a point on the Jordan curve, then $T_p\gamma \subset \C$ is a $1$-dimensional real subspace.
A direct calculation shows that
\[ T_{(p,0)}T_1 = T_p\gamma \times \{ 0 \} \oplus \{ 0 \} \times e^{i \phi} T_p\gamma \,\,\, {\rm and} \,\,\, T_{(p,0)}T_2 = T_p\gamma \times \{ 0 \} \oplus \{ 0 \} \times T_p\gamma, \]
so
\[ T_{(p,0)}T_1 \cap T_{(p,0)}T_2 = T_p \gamma \times \{ 0 \} = T_{(p,0)} (\gamma \times \{0\}), \]
and the intersection along $\gamma \times \{ 0 \}$ is clean, as required.
\end{proof}

\noindent
{\bf A surgered immersed Lagrangian torus.}
Because $T_1$ and $T_2$ intersect cleanly along $\gamma \times \{0\}$, a version of the Weinstein neighborhood theorem due to Po\'zniak \cite[Proposition 3.4.1]{pozniak} implies that we can select coordinates $(x_1,y_1,x_2,y_2)$ in a neighborhood $N \approx (\bR/\bZ) \times \bR^3$ of $\gamma \times \{0\} $ such that
\begin{itemize}
\item
$\omega = d x_1 \wedge d y_1 + d x_2 \wedge d y_2$,
\item
$T_1 \cap N = \{ y_1 = y_2 = 0 \}$, and
\item
$T_2 \cap N = \{ y_1 = x_2 = 0 \}$.
\end{itemize}
We smooth the intersection of $T_1$ and $T_2$ in $N$ as suggested by Figure \ref{fig:smoothing} and let $T$ denote the result.
The tangent plane to $T$ at a point in $N$ is spanned by $\del/\del x_1$ and a vector of the form $a \cdot \del/\del x_2 + b \cdot \del/\del y_2$, which are $\omega$-orthogonal.
Thus, $T$ is an immersed Lagrangian torus in $(\C^2,\omega)$, and its set of self-intersections equals $P$, which parametrizes the set of inscriptions of $Q$ in $\gamma$.

\begin{figure}
\includegraphics[width=1.5in]{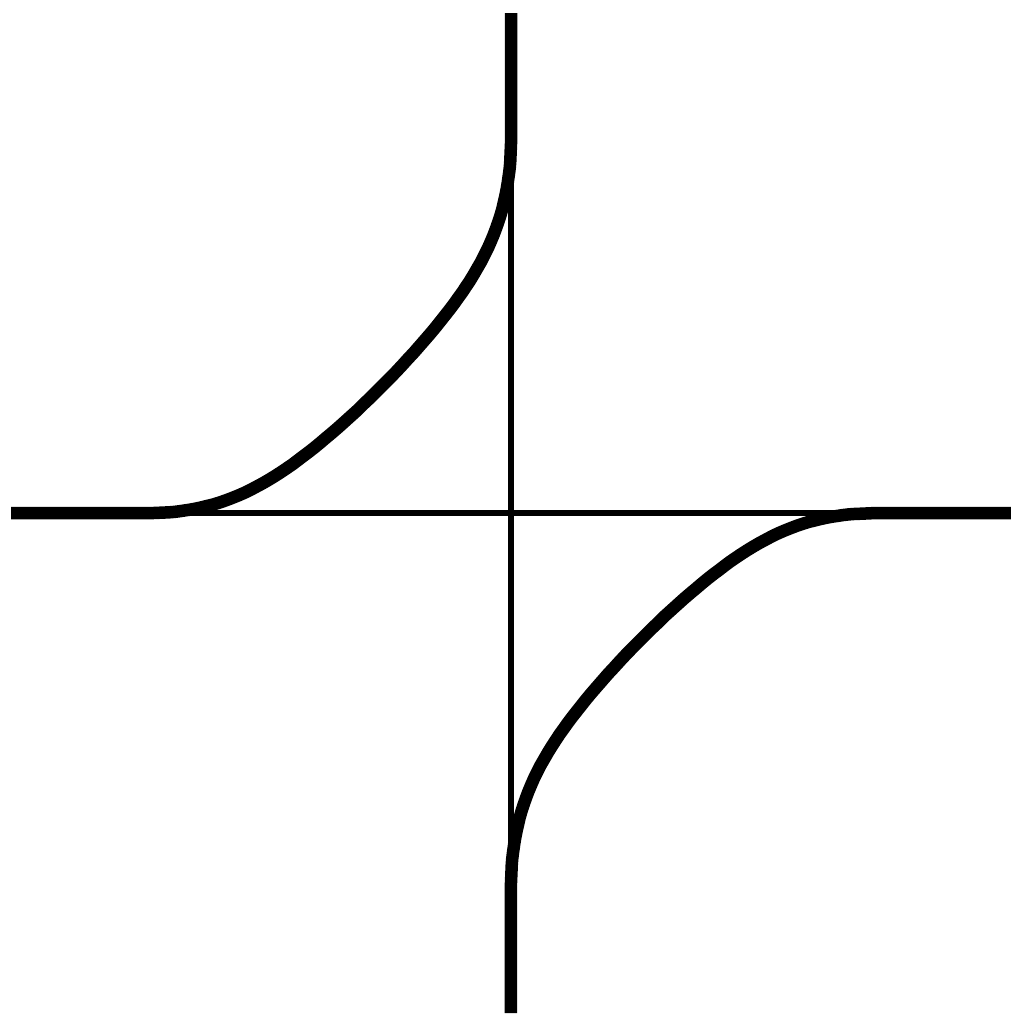}
\put(-10,60){$x_2$}
\put(-50,100){$y_2$}
\put(-44,58){$T_1$}
\put(-66,37){$T_2$}
\put(-75,75){$T$}
\caption{Cross-section of smoothing in the $x_1 = \mathrm{constant}$, $y_1 = 0$ plane.}
\label{fig:smoothing}
\end{figure}
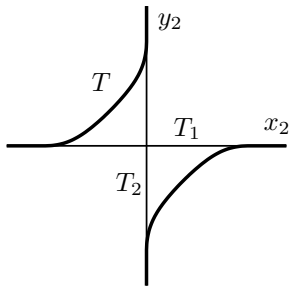

\noindent
{\bf The minimum Maslov number.}
Equip $\C^n$ with a product symplectic form $\omega_0 = \sum_{i=1}^n c_i \cdot dz_i \wedge~d\overline{z}_i$.
An immersed Lagrangian submanifold $i : L \to (\C^n,\omega_0)$ has a Maslov class $\mu \in H^1(L;\bZ)$, given as follows (cf. \cite[pp.117-118]{mcduffsalamon}).
The tangent planes to $i(L)$ along the image of an embedded loop $\alpha \subset L$ determine a loop $\alpha^\sharp$ in $\mathcal{L}(\omega_0)$, the Grassmannian of Lagrangian $n$-planes in $(\C^n,\omega_0)$.
The Maslov index of $\alpha$ is the value $\mu([\alpha]) := [\alpha^\sharp] \in H_1(\mathcal{L}(\omega_0);\bZ) \approx \bZ$, and the minimum Maslov number of $L$ is the non-negative integer $m(L)$ such that $\mu(H_1(L;\bZ)) = m(L) \cdot \bZ$.

\begin{prop*}
	\label{prop:maslovnumberis4}
	The minimum Maslov number of $T$ is $4$.
\end{prop*}

\begin{proof}
Orienting $\gamma \subset \C$ counterclockwise, its Maslov index equals 2 with respect to $c \cdot dz \wedge d\overline{z}$.
Hence $\gamma \times \{\mathrm{pt.}\}$ and $\{\mathrm{pt.}\} \times \gamma$ both have Maslov index 2 in $\gamma \times \gamma$ with respect to the product form $\omega_r$.
Since their homology classes generate $H_1(\gamma \times \gamma;\bZ)$, we obtain $m(\gamma \times \gamma)=2$.
The diagonal loop $\{ (z,z) : z \in \gamma \}$ is homologous to their sum, so it has Maslov index 4 in $\gamma \times \gamma$ with respect to $\omega_r$.
Applying $R_\phi \circ F_s$ and $F_t$, we deduce that $\gamma \times \{0\}$ has Maslov index 4 in both $T_1$ and $T_2$ with respect to $\omega$ and that $m(T_1)=m(T_2)=2$.
Let $\delta$ denote a push-off of $\gamma \times \{0\}$ in $T_1$ away from the site of surgery.
A neighborhood of $\delta$ survives the surgery, so the Maslov index of $[\delta]$ in $T$ is 4 with respect to $\omega$.

Next, select oriented loops $\lambda_1 \subset T_1$, $\lambda_2 \subset T_2$, and $\lambda \subset T$ such that $\lambda_1 \cup \lambda_2$ and $\lambda$ coincide outside the neighborhood $N$ above and meet it in a single slice $x_1 = \mathrm{constant}$, $y_1 = 0$, as displayed in Figure \ref{fig:smoothing}.
The tangent planes to $T \cup T_1 \cup T_2$ along the difference 1-cycle $\lambda - \lambda_1 - \lambda_2$ describe a nullhomotopic loop in $\mathcal{L}(\omega)$.
Consequently, $[\lambda^\sharp] = [\lambda_1^\sharp] + [\lambda_2^\sharp] \in H_1(\mathcal{L}(\omega);\bZ) \approx \bZ$.
The class $[\lambda_j]$ completes to a basis of $H_1(T_j;\bZ)$ with $[\gamma \times \{0\}]$ for $j=1,2$.
Since $m(T_j) = 2$ and $[\gamma \times \{0\}]$ has Maslov index $4$ in $T_j$, $j=1,2$, it follows that $[\lambda_j]$ has Maslov index $2 \pmod 4$, $j=1,2$.
Therefore, the Maslov index of $[\lambda]$ in $T$ is a multiple of 4.

Since $[\delta]$ and $[\lambda]$ form a basis for $H_1(T;\bZ)$, it follows that $m(T)=4$.
\end{proof}

%% file: sections/acknowledgements.tex
We thank Mohammed Abouzaid, Lev Buhovsky, Joe Johns, and Leonid Polterovich for useful discussions and Lenny Ng for drawing our attention to a key result from \cite{euclid}.